\documentclass{amsart}

\newtheorem{theorem}{Theorem}[section]

\newtheorem{lemma}[theorem]{Lemma}
\newtheorem{prop}[theorem]{Proposition}

\newtheorem*{maintheorem}{Main Theorem}

\theoremstyle{remark}

\numberwithin{equation}{section}

   \usepackage[dvips]{graphicx}
   \usepackage{psfrag}

\begin{document}

\title{Attractors and orbit-flip homoclinic orbits for star flows}

\author{C. A. Morales}
\address{Instituto de Matem\'atica,
Universidade Federal do Rio de Janeiro,
P. O. Box 68530,
21945-970 Rio de Janeiro, Brazil}
\email{morales@impa.br}
\thanks{Partially supported by CNPq, FAPERJ and PRONEX-Brazil.}






\begin{abstract}
We study star flows on closed $3$-manifolds and prove that they either have a finite number of attractors
or can be $C^1$ approximated by vector fields with orbit-flip homoclinic orbits.
\end{abstract}

\maketitle



\section{Introduction}

\noindent
The notion of attractor deserves a fundamental place
in the modern theory of dynamical systems.
This assertion, supported by the nowadays classical theory of turbulence \cite{rt},
is enlightened by the recent {\em Palis conjecture} \cite{pa} about the abundance of dynamical systems
with finitely many attractors
absorbing most positive trajectories.
If confirmed such a conjecture means the understanding of a
great part of dynamical systems in the sense of their long-term behaviour.

Here we attack a problem which goes straight to the Palis conjecture:
The finitude of the number of attractors for a given
dynamical system. Such a problem have been solved positively under certain circunstances.
For instance we have the work by Lopes \cite{lo} who,
based upon early works by Ma\~n\'e \cite{Ma} and extending previous ones by Liao \cite{li} and Pliss \cite{pli},
studied the structure of the $C^1$ structural stable diffeomorphisms
and proved the finitude of attractors for such diffeomorphisms.
His work was largely extended by Ma\~n\'e himself
in the celebrated solution of the $C^1$ stability conjecture \cite{Ma2}.
On the other hand, the japanesse researchers
S. Hayashi \cite{haa} and N. Aoki \cite{a} studied
the {\em star diffeomorphisms}, i.e., diffeomorphisms which cannot be $C^1$ approximated by ones
with nonhyperbolic periodic points, and proved that they
are Axiom A and so with only a finite number of attractors.
Their investigation triggered the study of the {\em star flows}, i.e.,
vector fields which cannot be $C^1$ approximated by ones with nonhyperbolic closed orbits.
Indeed, although it was known from the very beginning that
these flows are not necessarily Axiom A
\cite{abs}, \cite{gu}, \cite{GW}, the aforementioned works by Liao and Pliss
proved that they display finitely many attracting closed orbits.

A progress toward understanding star flows was tackled in 2003 by the author in collaboration with Pacifico \cite{mpa1}.
Indeed, these authors proved on closed $3$-manifolds that, except in a meager set,
all such flows are singular-Axiom A and so with
only a finite number of attractors.
Soon later the chinesse authors Gan and Wen \cite{gw}
extended the Aoki-Hayashi's conclusion to nonsingular star flows on closed manifolds
impliying that these flows has a finite number of attractors too.
In light of these works it seems quite promissing to prove the finiteness of the number of attractors
for star flows in any closed manifold.

In this paper we shall provide a result which though partial provides
an insight for a positive solution of this problem.
Basically, we present the so-called orbit-flip homoclinic orbits as obstruction for the finiteness of
attractors of star flows on closed $3$-manifolds.
More precisely, we show that a star flow on a closed $3$-manifold either has a finite number of attractors
or can be $C^1$ approximated by vector fields exhibiting orbit-flip homoclinic orbits.
{\em Orbit-flip homoclinic orbits} are very rich dynamical structures which have been studied during decates
\cite{c}, \cite{hok}, \cite{kko}, \cite{mpa3}, \cite{n}, \cite{sstc}, \cite{t}.
Let us state this result in a precise way.

Hereafter $M$ will denote a compact connected boundaryless Riemannian manifold of dimension $n\geq2$
(a {\em closed $n$-manifold} for short).
We shall consider a $C^1$ vector field $X$ in $M$ together with
its induced one-parameter group $X_t$, $t\in\mathbb{R}$,
the so-called flow of $X$.
The space of $C^1$ vector fields in $M$ comes equipped with the {\em $C^1$-topology} which,
roughly speaking,
measures the distance between vector fields and their corresponding derivatives.

The long-time behavior of a point
$x\in M$ is often analyzed through its {\em omega-limit set}
$$
\omega(x)=\left\{y\in M:y=\lim_{n\to\infty}X_{t_n}(y)\mbox{ for some sequence }t_n\to\infty\right\}.
$$
A compact invariant set is {\em transitive} if it coincides with the omega-limit set of one of its points,
whereas, in this work, an {\em attractor} will be a transitive set of the form
$$
A=\bigcap_{t>0}X_t(U)
$$
for some neighborhood $U$ of it.
The most representative example of attractors
are the {\em sinks}, that is, hyperbolic closed orbits
of maximal Morse index. Sometimes we use the term {\em source}
referring to a sink for the time reversal vector field $-X$.

A {\em homoclinic orbit}
is a regular (i.e. nonsingular) trajectory $\Gamma=\{X_t(q):t\in\mathbb{R}\}$
which is biasymptotic to a singularity $\sigma$, namely,
$$
\lim_{t\to\pm\infty}X_t(q)=\sigma.
$$
We call it
{\em orbit-flip} as soon as the eigenvalues $\lambda_1,\lambda_2,\lambda_3$ of $\sigma$ are real,
satisfy the eigenvalue inequalities
$\lambda_2<0<\lambda_3<\lambda_1$
and $\Gamma\subset W^{uu}(\sigma)$
where $W^{uu}(\sigma)$, the {\em strong unstable manifold} \cite{hps},
is the unique invariant manifold of $X$ which
is tangent at $\sigma$ to the eigenspace associated to the eigenvalue $\lambda_1$
(c.f. Figure \ref{exxx}).

\begin{figure}[htb] 
\begin{center}
\includegraphics[scale=0.3]{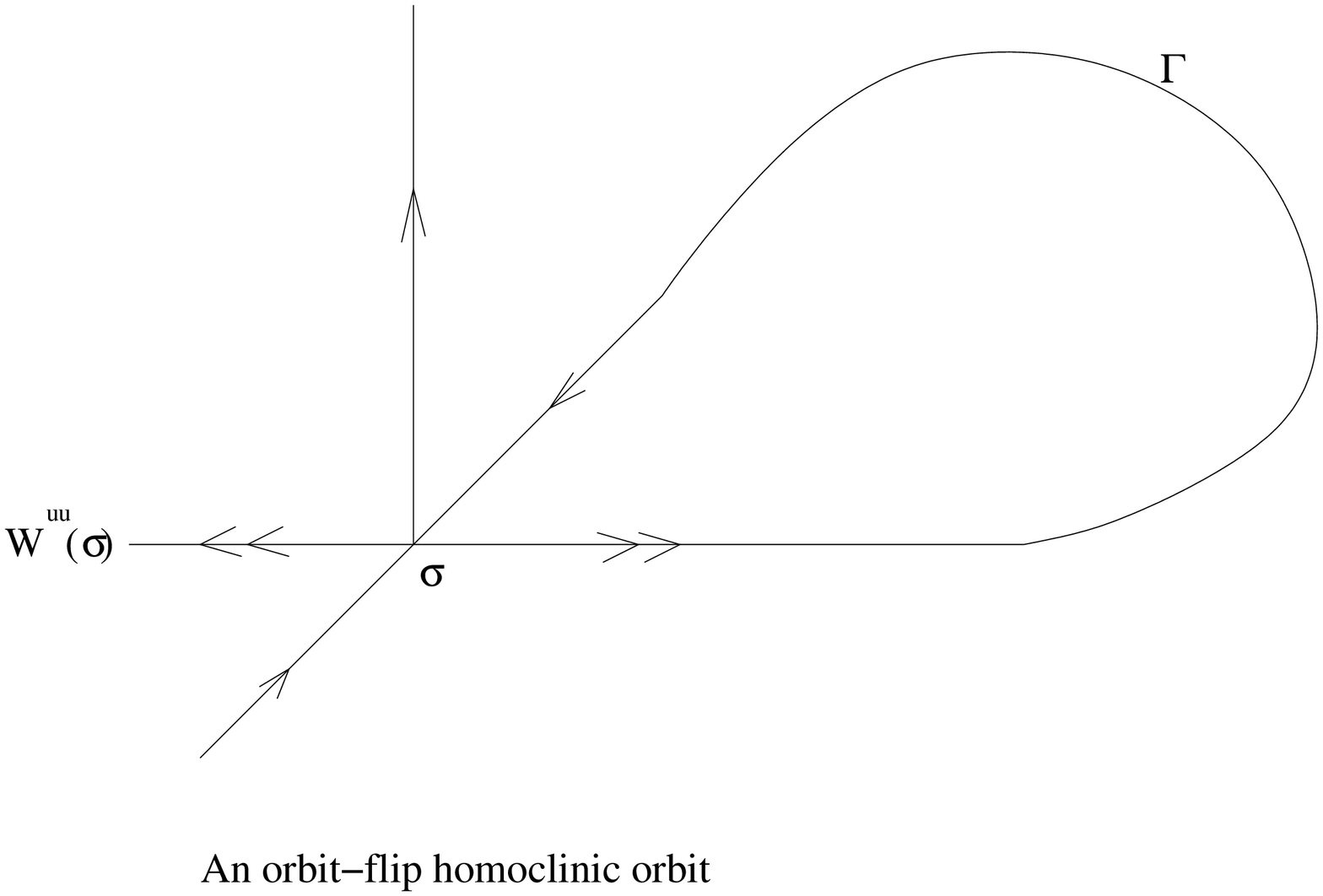}
\caption{\label{exxx}}
\end{center}
\end{figure}

With such definitions and notations we can state our result.

\begin{maintheorem}
 \label{thA}
A star flow on a closed $3$-manifold either has a finite number of attractors
or can be $C^1$ approximated by vector fields exhibiting orbit-flip homoclinic orbits.
\end{maintheorem}

The proof relies on recent results in the theory of star flows
\cite{gw} together with some techniques resembling those in \cite{m}.

The Main Theorem motivates the obvious question if
star flows which
can be $C^1$ approximated by vector fields with orbit-flip homoclinic orbits
exist on any closed $3$-manifold.
Actually this is true but, as the reader can see by himself \cite{li}, \cite{m}, \cite{mpa3}, \cite{pli},
the set of such flows constitute a meager subset of star flows.
We therefore conclude that every closed $3$-manifold
comes equipped with an open and dense subset of
star flows, all of whose elements have a finite number of attractors.
However, it is worth noting that we can obtain exactly the same conclusion
by making use of \cite{m} and \cite{mpa1}.

Another question is if, in the statement of the Main Theorem, we can replace
the finitely many attractor's option by the stronger property of being singular-Axiom A
(in the sense of \cite{mpa1}).
Unfortunately, such a question has negative answer as we can easily find
star flows in the $3$-sphere which neither are singular-Axiom A nor can be $C^1$ approximated by vector fields
with orbit-flip homoclinic loops.
Finally let us mention that, in the statement of the Main Theorem,
we can replace the term attractor by that of Lyapunov stable omega-limit set
(in the spirit of \cite{mop}).

\section{Proof}

\noindent
We denote by $\|\cdot\|$ the norm induced by a Riemannian metric in $M$
and by $m(\cdot)$ its corresponding minimum norm.
Given a $C^1$ vector field $X$ with flow $X_t$ in $M$ we denote by
$Sing(X,U)$ the set of singularities of $X$ in $U$ and we write $Sing(X)=Sing(X,M)$.
Likewise, the union of the periodic orbits of $X$ is denoted by $Per(X)$. The elements of
$Per(X)$ will be called {\em periodic points}.
A subset $\Lambda\subset M$ is called {\em invariant} if
$X_t(\Lambda)=\Lambda$ for all $t\in\mathbb{R}$.
A compact invariant set $\Lambda$ is {\em hyperbolic}
if there are a tangent bundle decomposition
$T_\Lambda M=\hat{E}^s_\Lambda\oplus E^X_\Lambda\oplus \hat{E}^u_\Lambda$
and positive constants $K,\lambda$ such that
\begin{itemize}
 \item
$\hat{E}^s_\Lambda$ is contracting, i.e.,
$$
\|DX_t(x)/\hat{E}^s_x\|\leq Ke^{-\lambda t},
\quad\quad\forall x\in \Lambda, t\geq 0.
$$
\item
$\hat{E}^u_\Lambda$ is expanding,
i.e.,
$$
m(DX_t(x)/\hat{E}^u_x)\geq K^{-1}e^{\lambda t},
\quad\quad\forall x\in \Lambda, t\geq 0.
$$
\item
$E^X_\Lambda$ is the subbundle generated by $X$ in $T_\Lambda M$.
\end{itemize}
Recall that a closed orbit $O$ (i.e. a periodic orbit or singularity)
is hyperbolic if it does as a compact invariant set.
A periodic point is hyperbolic if its corresponding orbit is.
We say that a compact invariant set is {\em nontrivial}
if it does not reduce to a single closed orbit.

For any given index $i$ in between $0$ and $n-1$ we denote by
$Per_i(X)$ the union of the hyperbolic periodic orbits $O$ of $X$ with $dim(\hat{E}^s_x)=i$
for some (and hence for all) $x\in O$.
As pointed out earlier by Wen \cite{w} we
can extended this set to include periodic orbits for nearby vector fields.
More precisely, we denote by $Per^*_i(X)$ the {\em $i$-preperiodic set} of $X$
consisting of those $x\in M$ for which there are sequences $X_k$ and $x_k$
of vector fields and points in $Per_i(X_k)$
such that $X_k\to X$ and $x_k\to x$.
We shall also use the notion of
{\em fundamental $i$-limit}
which are limits (in the Hausdorff metric) of sequences
of hyperbolic periodic orbits $O_n\subset Per_i(X_n)$
of vector fields $X_n\to X$.

Now we state four technical lemmas the first of which is Lemma 3.4 in \cite{gw}:

\begin{lemma}
 \label{gw1}
If $X$ is a star flow on a closed $n$-manifold and $\Lambda$ is a fundamental $i$-limit
of $X$ with $Sing(X,\Lambda)=\emptyset$,
then $\Lambda$ is a sink or a source depending on whether $i=n-1$ or $i=0$.
\end{lemma}

Denote by $N\to M\setminus Sing(X)$ the vector bundle with fiber
$N_x=\{v\in T_xM:v\, \bot\,X(x)\}$.
We define the {\em linear Poincar\'e flow} $P_t:N\to N$
by
$$
P_t=\pi\circ DX_t,
$$
where $\pi:TM \to N$ stands for orthogonal projection.
A {\em $P_t$-invariant splitting}  over an invariant set $\Lambda\subset M\setminus Sing(X)$
is a direct sum
$
N_\Lambda=\Delta^{1}_\Lambda\oplus \Delta^{2}_\Lambda
$
such that $P_t(\Delta^{1}_x)=\Delta^{1}_{X_t(x)}$ and $P_t(\Delta^{2}_x)=\Delta^{2}_{X_t(x)}$
for all $x\in \Lambda$ and $t\in\mathbb{R}$.

We shall use the following Doering's criterium for hyperbolicity \cite{d}:
A compact invariant set $\Lambda$ with $Sing(X,\Lambda)=\emptyset$ is
hyperbolic if and only if there is a
$P_t$-invariant splitting  $N_\Lambda=\Delta^{1}_\Lambda\oplus \Delta^{2}_\Lambda$
over $\Lambda$ such that $\Delta^{1}_\Lambda$ is {\em contracting} and
$\Delta^{2}_\Lambda$ is {\em expanding}, i.e.,
there are positive constants $K,\lambda$ satisfying
$$
\|P_t(x)/\Delta^{1}_x\|\leq Ke^{-\lambda t}
\quad\mbox{and }
\quad
m(P_t(x)/\Delta^{2}_x)\geq K^{-1}e^{\lambda t}
\quad\forall x\in \Lambda, \forall t\geq 0.
$$
A {\em dominated splitting} for $P_t$ over $\Lambda$
is a $P_t$-invariant splitting
$
N_\Lambda=\Delta^{-}_\Lambda\oplus \Delta^{+}_\Lambda
$
for which there are positive constants $K,\lambda$ satisfying
$$
\frac{\|P_t(x)/\Delta^{-}_x\|}{m(P_t(x)/\Delta^{+}_x)}\leq Ke^{-\lambda t},
\quad\forall (x,t)\in \Lambda\times \mathbb{R}^+.
$$
A {\em dominated $\rho$-splitting} for $P_t$ over $\Lambda$ is a dominated splitting
$N_\Lambda=\Delta^{-}_\Lambda\oplus \Delta^{+}_\Lambda$
such that $dim(\Delta^{-}_x)=\rho$, $\forall x\in\Lambda$.

The following is Lemma 3.10 in \cite{gw}.

\begin{lemma}
 \label{gw2}
Let $X$ a star flow on a closed $n$-manifold and
$\Lambda$ be a compact invariant set with $Sing(X,\Lambda)=\emptyset$
for which there is a dominated $\rho$-splitting
$N_\Lambda=\Delta^{-}_\Lambda\oplus \Delta^{+}_\Lambda$
for $P_t$ over $\Lambda$ with $1\leq\rho\leq n-2$.
If $\Delta^{-}_\Lambda$ is not contracting there is a fundamental
$r$-limit contained in $\Lambda$ with $r<\rho$.
Likewise, if $\Delta^{+}_\Lambda$ is not expanding
there is a fundamental $r$-limit contained in $\Lambda$ with $r>\rho$.
\end{lemma}

The proof of the following result can be obtained
as in Theorem 3.8 of \cite{wen} (see also the proof of Lemma 2.8 in \cite{glw}).

\begin{lemma}
\label{dom-pre}
If $X$ is a star flow, then
for every index $1\leq i\leq n-2$
there is a dominated $i$-splitting
$$
N_{Per_i^*(X)\setminus Sing(X)}=\Delta^{-}_{Per_i^*(X)\setminus Sing(X)}\oplus
\Delta^{+}_{Per_i^*(X)\setminus Sing(X)}
$$
for $P_t$ over $Per_i^*(X)\setminus Sing(X)$
such that
$$
\Delta^{-}_x=\pi_x(\hat{E}^{s}_x)
\quad \mbox{ and }\quad
\Delta^{+}_x=\pi^Y_x(\hat{E}^{u}_x),
\quad\forall x\in Per_i(X),
$$
where $T_xM=\hat{E}^{s}_x\oplus E^X_x\oplus \hat{E}^{u}_x$ is the corresponding hyperbolic splitting
along the orbit of $x$.
\end{lemma}

The following lemma is Theorem B in \cite{gw}. Denote by $Cl(\cdot)$ the closure operation.

\begin{lemma}
 \label{gw3}
If $X$ is a star flow on a closed manifold, then every compact invariant set
$\Lambda\subset Cl(Per(X))$ with $Sing(X,\Lambda)=\emptyset$ is hyperbolic.
\end{lemma}

These lemmas will be used to analyze attractors
for star flows on closed $3$-manifold.
To start with we extend the conclusion of Lemma \ref{gw3} to all such attractors.

\begin{prop}
 \label{omega-hyperbolic}
If $X$ is a star flow on a closed $3$-manifold, then every attractor $A$ of $X$
with $Sing(X,\Lambda)=\emptyset$ is hyperbolic.
\end{prop}

\begin{proof}
First we show that $A\subset Per_1^*(X)$ unless $A$ is a sink or a source. 
Indeed, if $A\not\subset Per_1^*(X)$ we can select $y\in A\setminus Per^*_i(X)$.
As $A$ has no singularities and $y\in A$ we have that
$y$ is a regular point (i.e. $X(y)\neq0$).
Then, it follows from the Pugh's closing lemma \cite{p} that $y\in Per_0^*(X)\cup Per_2^*(X)$
and so there exist a fundamental $i$-limit
with $i=0,2$ intersecting $A$. As $A$ is an attractor we conclude that
such a fundamental $i$-limit is contained in $A$.
Therefore, by Lemma \ref{gw1}, it would exist
a sink or a source contained in $A$. In such a case $A$ is a sink or a source.
Then, we can assume that $A\subset Per^*_1(X)$. So, Lemma \ref{dom-pre}
implies that there is a dominated $1$-splitting
$N_{A}=\Delta^{-}_{A}\oplus \Delta^{+}_{A}$ for $P_t$ over $A$.
If subbundle $\Delta^-_A$ were not contracting it would exist
a fundamental $0$-limit in $A$ in virtue of Lemma \ref{gw2}.
Therefore $A$ is a source and so hyperbolic.
Hence we can assume that $\Delta^{-}_{A}$ is contracting and analogously that
$\Delta^{+}_{A}$ is expanding.
Then, $A$ is hyperbolic by the Doering's criterium.
\end{proof}

The following elementary lemma will be used to prove Proposition \ref{thha}.

\begin{lemma}
 \label{A}
For every $\epsilon>0$ there is $\delta>0$ such that if
$c:[a,b]\subset [-\epsilon,\epsilon]\to [-\epsilon,\epsilon]$ is a $C^1$ map
satisfying
\begin{enumerate}
 \item[(i)]
$|c'(t)|\leq \frac{1}{6}$ for all $t\in [a,b]$;
\item[(ii)]
$|c(t_0)|\leq \delta$ for some $t_0\in [a,b]$;
\item[(iii)]
$(a,c(a)),(b,c(b))\in\partial([-\epsilon,\epsilon]^2)$,
\end{enumerate}
then $a=-\epsilon$, $b=\epsilon$ and $|c(\pm \epsilon)|< \epsilon$.
\end{lemma}

\begin{proof}
We take $\delta=\frac{\epsilon}{3}$.
Without loss of generality we can assume that $t_0$ in (ii) belongs to $]a,b[$.
If $-\epsilon<a$ then condition (iii) implies $c(a)=\pm \epsilon$.
On the other hand, condition (i) and the mean value theorem imply
$|c(t_0)-c(a)|\leq \frac{1}{6}|t_0-a|\leq \frac{\epsilon}{3}$
thus (ii) yields
$\epsilon\leq \frac{\epsilon}{3}+\frac{\epsilon}{3}=\frac{2\epsilon}{3}$ which is absurd.
Therefore $a=-\epsilon$ and analogously $b=\epsilon$. The same computation shows
$|c(\pm \epsilon)|\leq\frac{\epsilon}{3}$.
\end{proof}

Hereafter we will use the standard
stable and unstable manifold notation $W^s(\cdot)$, $W^u(\cdot)$ (c.f. \cite{hps}).

\begin{prop}
\label{thha}
Let $X$ be a star flow on a closed $3$-manifold and $\sigma\in Sing(X)$
be such that either
\begin{enumerate}
 \item[(iv)]
$dim(W^s(\sigma))=2$ or
\item[(v)]
$\sigma$ has three real eigenvalues $\lambda_2<0<\lambda_3<\lambda_1$
and $X$ cannot be $C^1$ approximated by vector fields with orbit-flip homoclinic orbits.
\end{enumerate}
Then, for every $x\in (W^s(\sigma)\setminus \{\sigma\})\cap Per^*_1(X)$
there is $\delta>0$ such that $d(A,x)>\delta$ for every nontrivial
hyperbolic attractor $A$ of $X$.
\end{prop}

\begin{proof}
Clearly $x\in Per^*_1(X)\setminus Sing(X)$,
and so, by Lemma \ref{dom-pre},
there is a dominated $1$-splitting
$N_{x}=\Delta^-_{x}\oplus\Delta^+_{x}$
for $P_t$.
It turns out that $\Delta^-_x=L_x\cap N_x$ where $L_x$ is either $T_xW^s(\sigma)$
or $T_xW^{sE}(\sigma)$ depending on whether
(iv) or (v) holds. Here $W^{sE}(\sigma)$ is the {\em extended stable manifold},
i.e., the invariant manifold tangent at $\sigma$ to the eigenspace
corresponding to the eigenvalues $\lambda_2,\lambda_3$ (c.f. \cite{hps}, \cite{sstc}).
Using this we can fix a cross-section $\Sigma=[-\epsilon,\epsilon]^2$
through $x=(0,0)$ of $X$ so that:
\begin{itemize}
 \item 
If (iv) holds then $W^s(\sigma)\cap \Sigma$ contains the graph $\gamma=\{(u(y),y):y\in [-\epsilon,\epsilon]\}$
of a $C^1$ map $u:[-\epsilon,\epsilon]\to [-\epsilon,\epsilon]$ with $u(0)=0$
(c.f. Figure \ref{me}-(a)).
\item
If (v) holds then there are $C^1$ maps
$u_1\leq u_2:[-\epsilon,\epsilon\to [-\epsilon,\epsilon]$
with $u_i(0)=u'_i(0)=0$ (for $i=1,2$) such that
$Per_1^*(X)\cap R=\emptyset$ where $R\subset \Sigma$
is the complement of the region of $\Sigma$ in between the graphs of $\gamma_1$ and $\gamma_2$
of $u_1$ and $u_2$ respectively (i.e. the complement of the shadowed region in Figure \ref{me}-(b)).
\end{itemize}

\begin{figure}[htb] 
\begin{center}
\includegraphics[scale=0.3]{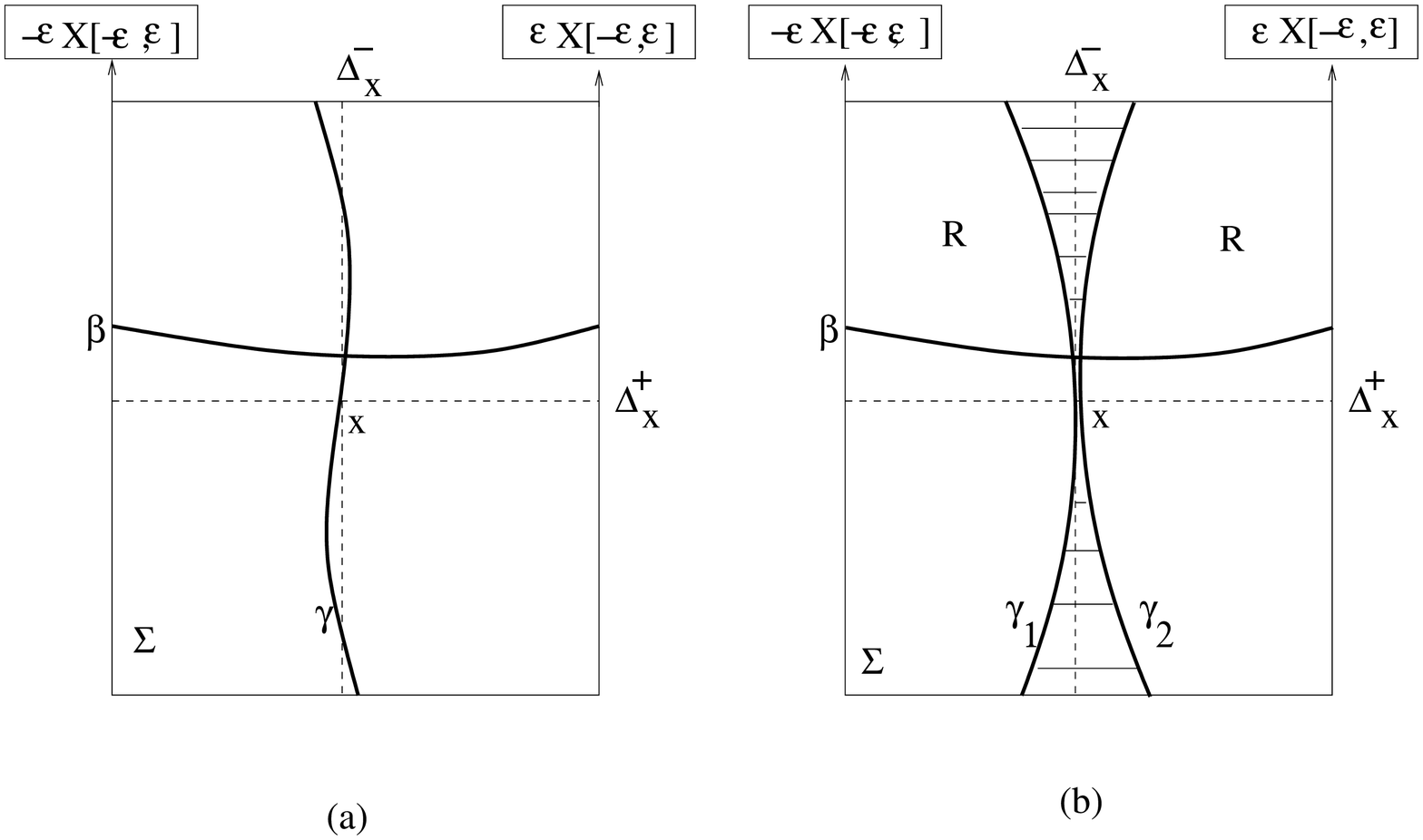}
\caption{\label{me}}
\end{center}
\end{figure}

Shrinking $\epsilon$ if necessary we can assume that
any $C^1$ map
$c:[a,b]\subset [-\epsilon,\epsilon]\to [-\epsilon,\epsilon]$
whose graph $\gamma=\{(t,c(t):t\in [a,b]\}$
is tangent to $\Delta^+_{Per^*_1(X)\cap \Sigma}$ satisfies hypothesis (i) in
Lemma \ref{A}, i.e., $|c'(t)|\leq \frac{1}{6}$ for all $t\in[a,b]$.

Now, take $\delta>0$ as in Lemma \ref{A} for such an $\epsilon$
and a nontrivial hyperbolic attractor $A$ with $d(A,x)\leq \delta$.
Then, using tubular flow box around $x$ we have that
there is $y\in A\cap \Sigma$ with $d(x,y)\leq \delta$.

Let $\beta$ be the connected component of $W^u(y)\cap \Sigma$ containing $y$.
Since $A$ is a nontrivial hyperbolic attractor
standard facts about hyperbolic sets (e.g. the local product structure \cite{hk})
imply that the end points of $\beta$ belong to $\partial\Sigma$. Moreover,
$A\subset Per_1^*(X)\setminus Sing(X)$ (by the shadowing lemma for flows \cite{hk})
and, since $\Delta^{+}_x=\pi_x(\hat{E}^{u}_x)$ for $x\in Per_1(X)$ (by Lemma \ref{dom-pre})
and the periodic orbits in $A$ are dense in $A$, we obtain that
$\beta$ is tangent to $\Delta^+_{Per^*_1(X)\cap \Sigma}$.
Then, $\beta$ is the graph of a $C^1$ map
$c:[a,b]\to [-\epsilon,\epsilon]$ with $c(t_0)=y$ for some $t_0\in (a,b)$, and so,
$c$ satisfies hypotheses (i) and (ii) of Lemma \ref{A}.
Additionally, since the end points of $\beta$ belong to $\partial\Sigma$ we also have
$c(a),c(b)\in \partial (\Sigma)$ and so $c$ also satisfies hypothesis (iii) of
Lemma \ref{A}. Then, Lemma \ref{A}
implies $a=-\epsilon$ and $b=\epsilon$ and $|c(\pm\epsilon)|<\epsilon$.
Consequently, $\beta$ joins $-\epsilon\times [-1,1]$ to $\epsilon\times [-1,1]$
as indicated in Figure \ref{me}.

If (iv) holds, then $\beta$ (which is contained in $A$) intersects
$\gamma$ (which is contained in $W^s(\sigma)$) whence $\sigma\in A$ which is absurd
since $A$ is a nontrivial hyperbolic attractor.
Therefore, (v) holds and so $\beta\cap R\neq\emptyset$
yielding $Per_1^*(X)\cap R\neq \emptyset$ again an absurdity.
These contradictions prove the result.
\end{proof}

Now we prove the following key result.

\begin{prop}
 \label{finiteness}
Let $X$ be a star flow with singularities on a closed $3$-manifold
which cannot be $C^1$ approximated by vector fields with orbit-flip homoclinic orbits.
Then, there is a neighborhood $U$ of $Sing(X)$ such that if $A$ is an attractor of $X$ then $A\cap U\neq\emptyset$
if and only if $Sing(X,A)\neq\emptyset$.
\end{prop}

\begin{proof}
Otherwise there is a sequence of attractors $A_n$ with $Sing(X,A_n)=\emptyset$ and
$\sigma\in \Lambda\cap Sing(X)$,
where
$$
\Lambda=Cl\left(\bigcup_nA_n\right).
$$
Since star flows have finitely many sinks (\cite{li}, \cite{pli})
we can assume that each $A_n$ is nontrivial and they are all hyperbolic by Proposition \ref{omega-hyperbolic}.
It follows that every $A_n$ is a nontrivial hyperbolic attractor and so $\Lambda\subset Per_1^*(X)$.

We clearly have that $\sigma$ is neither a sink nor a source
(otherwise it could not be accumulated by periodic orbits
which is the case for $\sigma$).
So, we can order its eigenvalues $\lambda_1,\lambda_2,\lambda_3$
in a way that either $\lambda_2$ or $\lambda_1$ is real and, in each case,
$$
Re(\lambda_2)\leq Re(\lambda_3)<0<\lambda_1
\quad\mbox{ or }\quad
\lambda_2<0<Re(\lambda_3)\leq Re(\lambda_1)
$$
with $Re(\cdot)$ denoting real part.

In the first case $\sigma$ clearly satisfies hypothesis (iv) of Proposition \ref{thha}.
In the second we must have that both $\lambda_3$ and $\lambda_1$ are real
(otherwise the dominated $1$-splitting claimed to exist
in Lemma \ref{dom-pre} would not exist) and, since $X$ cannot be approximated by vector fields
with orbit-flip homoclinic loops, we still have $\lambda_3<\lambda_1$.
In other words, in such a case $\sigma$ satisfies hypothesis (v) of Proposition \ref{thha}.
On the other hand, in both cases it is certainly possible to find
$x\in (W^s(\sigma)\setminus \{\sigma\})\cap \Lambda$.
In particular, $x\in Per_1^*(X)\setminus Sing(X)$ and so,
by Proposition \ref{thha}, there is $\delta>0$ such that
$d(A_n,x)\geq \delta$ for all $n$.
But this is clearly impossible due to the definition of $\Lambda$ so
the result is true.
\end{proof}

\begin{proof}[Proof of the Main Theorem]
Let $X$ be a star flow on a closed $3$-manifold
which cannot be $C^1$ approximated by vector fields exhibiting orbit-flip homoclinic orbits.
We can assume without any loss of generality that $X$ has singularities (if not we apply \cite{gw}).
Suppose by contradiction that it has infinitely many distinct attractors $A_n$, $n\in\mathbb{N}$.
Since $X$ has finitely many singularities and sinks, and, since
the attractors are pairwise disjoint, we can assume that
each $A_n$ is not a sink and satisfies
$Sing(X,A_n)=\emptyset$. In particular, each $A_n$ is a nontrivial hyperbolic attractor
by Proposition \ref{omega-hyperbolic}.
Moreover, by Proposition \ref{finiteness}, there is a neighborhood $U$ of
$Sing(X)$ such that $A_n\cap U=\emptyset$ for all $n$.
It follows that the closure
$Cl\left(\bigcup_n A_n\right)$ has no singularities.
Since each $A_n$ is a nontrivial hyperbolic attractor we have $A_n\subset Cl(Per(X))$
and so $Cl\left(\bigcup_n A_n\right)$ is also a compact invariant set in $Cl(Per(X))$.
Applying
Lemma \ref{gw3} we conclude that $Cl\left(\bigcup_n A_n\right)$ is a hyperbolic set.
However, as is well known, hyperbolic sets
contains only a finite number of attractors which is certainly not the case
for $Cl\left(\bigcup_n A_n\right)$.
We obtain so a contradiction which proves the result.
\end{proof}

\end{document}